\documentclass[11pt,reqno]{amsart}
\usepackage{color}

\usepackage{amsfonts,amscd}
\usepackage{amssymb}
\usepackage{url}
\usepackage[english]{babel}

\theoremstyle{plain}
\newtheorem{theorem}                 {Theorem}      [section]
\newtheorem{proposition}  [theorem]  {Proposition}

\theoremstyle{definition}
\newtheorem{example}      [theorem]  {Example}
\newtheorem{remark}       [theorem]  {Remark}
\newtheorem{definition}   [theorem]  {Definition}

\numberwithin{equation}{section}
\def \R{{\mathbb R}}
\def \rn{{\mathbb R}}
\def \s{{\mathbb S}}
\def \cn{{\mathbb C}}
\def \zn{{\mathbb Z}}

\def \H{{\mathbb H}}

\def \hyp{\mathcal H}

\def \nab#1#2{\hbox{$\nabla$\kern -.3em\lower 1.0 ex
    \hbox{$#1$}\kern -.1 em {$#2$}}}

\def \SL2{\widetilde{\text{\bf SL}}_{2}(\rn)}

\def \SO#1{\text{\bf SO}(#1)}

\def \SU#1{\text{\bf SU}(#1)}

\def \Sp#1{\text{\bf Sp}(#1)}

\DeclareMathOperator{\Div}{div}

\numberwithin{equation}{section}

\begin{document}

\title[Biharmonic functions on spheres and hyperbolic spaces]
{Biharmonic functions on spheres \\  and hyperbolic spaces}


\author{Sigmundur Gudmundsson}

\address{Mathematics, Faculty of Science\\ University of Lund\\
Box 118, Lund 221\\
Sweden}
\email{Sigmundur.Gudmundsson@math.lu.se}

\begin{abstract}
We construct new explicit proper $r$-harmonic functions on the standard
$n$-dimensional hyperbolic spaces $\H^n$ and spheres $\s^n$ for any
$r\ge 1$ and $n\ge 2$.
\end{abstract}

\subjclass[2010]{31B30, 53C43, 58E20}

\keywords{Biharmonic functions, spheres, hyperbolic spaces}

\maketitle

\section{Introduction}

The biharmonic equation is a fourth order partial differential equation which arises in areas of continuum mechanics, including elasticity theory and the solution of Stokes flows. The literature on biharmonic functions is vast, but usually the domains are either surfaces or open subsets of flat Euclidean space $\rn^n$.

Recently, new explicit biharmonic functions were constructed on the classical compact simple Lie groups $\SU n$, $\SO n$ and $\Sp n$, see \cite{Gud-Mon-Rat-1}.  This gives solutions on the 3-dimensional round sphere $\s^3\cong\SU 2$ and the standard hyperbolic space $\H^3$ via a general duality principle. In the literature we have only found explicit proper biharmonic functions from spheres and hyperbolic spaces of dimensions 2 and 3.   For this see the papers \cite{Caddeo}, \cite{Gud-Mon-Rat-1} and \cite{Gud-13}.

The aim of this work is to extend the investigation to higher dimensional spheres $\s^n$ and hyperbolic spaces $\H^n$. We construct a wide collection of new proper biharmonic functions from these spaces of any dimension $n\ge 2$.

The $n$-dimensional hyperbolic space can be modelled in several different ways.  The classical upper-half space model $\H^n$ is the most useful for our purposes.  First we construct a wealth of proper $r$-harmonic functions on the hyperbolic upper-half space $\H^n$, see Theorem \ref{theo:upper}.  Then we formulate our solutions in terms of the standard one-sheeted hyperboloid $\hyp^n$ as a hypersurface of the corresponding Minkowski space $M^{n+1}$, see Theorem \ref{theo:hyperboloid}.  Finally we  employ a general duality principle, between $\hyp^n$ and the standard $n$-dimensional sphere $\s^n$, to construct proper $r$-harmonic functions on $\s^n$, see Theorem \ref{theo:sphere}.

\section{Preliminaries}\label{section-preliminaries}

Let $(M,g)$ be a smooth m-dimensional manifold equipped with a Riemannian metric $g$.  We complexify the tangent bundle $TM$ of $M$ to $T^{\cn}M$ and extend the metric $g$ to a complex-bilinear form on $T^{\cn}M$.  Then the gradient $\nabla f$ of a complex-valued function $f:(M,g)\to\cn$ is a section of $T^{\cn}M$.  In this situation, the well-known {\it linear} Laplace-Beltrami operator (alt. tension field) $\tau$ on $(M,g)$ acts locally on $f$ as follows
\begin{equation}\label{equa:tension}
\tau(f)=\Div (\nabla f)=\sum_{i,j=1}^m\frac{1}{\sqrt{|g|}} \frac{\partial}{\partial x_j}
\left(g^{ij}\, \sqrt{|g|}\, \frac{\partial f}{\partial x_i}\right).
\end{equation}

\begin{definition}\label{definition-proper-r-harmonic}
For an integer $r>0$ the iterated Laplace-Beltrami operator $\tau^r$ is given by
$$\tau^{0} (f)=f\ \ \text{and}\ \ \tau^r (f)=\tau(\tau^{(r-1)}(f)).$$
We say that a complex-valued function $f:(M,g)\to\cn$ is
\begin{enumerate}
\item[(a)] {\it $r$-harmonic} if $\tau^r (f)=0$, and
\item[(b)] {\it proper $r$-harmonic} if $\tau^r (f)=0$ and $\tau^{(r-1)} (f)$ does not vanish identically.
\end{enumerate}
\end{definition}

It should be noted that the {\it harmonic} functions are exactly $r$-harmonic for $r=1$
and the {\it biharmonic} functions are the $2$-harmonic ones.
In some texts, the $r$-harmonic functions are also called {\it polyharmonic} of order $r$.
\vskip .2cm

In the paper \cite{Gud-Mon-Rat-1} the authors develop an interesting connections between the theory of $r$-harmonic functions and the notion of harmonic morphisms.  More specifically, we recall that a map $\pi:(\hat M,\hat g)\to(M,g)$ between two semi-Riemannian manifolds is a {\it harmonic morphism} if it pulls back germs of harmonic functions to germs of harmonic functions. The standard reference on this topic is the book \cite{BW-book} of Baird and Wood.   We also recommend the updated online bibliography \cite{Gud-bib}.  Later on we will make use of the following result.

\begin{proposition}\cite{Gud-Mon-Rat-1}\label{prop:lift-tension}
Let $\pi:(\hat M,\hat g)\to (M,g)$ be a submersive harmonic morphism
from a semi-Riemannian manifold $(\hat M,\hat g)$ to a Riemannian manifold
$(M,g)$. Further let $f:(M,g)\to\cn$ be a smooth function and
$\hat f:(\hat M,\hat g)\to\cn$ be the composition $\hat f=f\circ\pi$.
If $\lambda:\hat M\to\rn^+$ is the dilation of $\pi$ then the tension
field satisfies
$$
\tau(f)\circ\pi=\lambda^{-2}\tau(\hat f)\ \ \text{and}
\ \ \tau^r(f)\circ\pi=\lambda^{-2}\tau(\lambda^{-2}\tau^{(r-1)}(\hat f))
$$
for all positive integers $r\ge 2$.
\end{proposition}

\section{The hyperbolic upper-half space $\H^n$}

In this section we construct new complex-valued proper $r$-harmonic functions
on the $n$-dimensional hyperbolic space $\H^n$ for any $r\ge 1$ and $n\ge 2$.
We model $\H^n$ as the hyperbolic upper-half space i.e. the
differentiable manifold $$\H^n=\{(t,x)|\ t\in\rn^+\ \text{and}\ x\in\rn^{n-1}\}$$
equipped with its standard Riemannian metric $ds^2$ satisfying
$$ds^2=\frac 1{t^2}\cdot(dt^2+dx_1^2+\dots +dx_{n-1}^2).$$
It is then a direct consequence of Equation (\ref{equa:tension})
that the corresponding Laplace-Beltrami operator $\tau$ satisfies
$$\tau(f)=t^2\cdot (\frac{\partial^2 f}{\partial x_1^2}+\cdots
+\frac{\partial^2 f}{\partial x_{n-1}^2})+t^2\cdot\frac{\partial^2 f}{\partial t^2}
-(n-2)\cdot t\cdot\frac{\partial f}{\partial t}.$$

\begin{theorem}\label{theo:basic-step}
Let the $n$-dimensional hyperbolic space $\H^n$ be modelled as the upper-half space i.e.
$\H^n=\rn^+\times\rn^{n-1}$.  Let $h:\rn^{n-1}\to\cn$ be a non-constant function harmonic
with respect to the Euclidean metric on $\rn^{n-1}$ and $p_1:\rn^+\to\cn$ be differentiable.
Then the function $f_1:\H^n\to\cn$ defined by $$f_1(t,x)=p_1(t)\cdot h(x)$$ is harmonic on $\H^n$
if and only if $p_1$ is of the form $p_1(t)=a_1+b_1\cdot t^{n-1}$, for some constants
$a_1,b_1\in\cn$.
\end{theorem}

\begin{proof}
We are assuming that $h:\rn^{n-1}\to\cn$ is a harmonic function with respect to
the Euclidean metric on $\rn^{n-1}$ i.e.
$$\frac{\partial^2 h}{\partial x_1^2}+\cdots +\frac{\partial^2 h}{\partial x_{n-1}^2}=0.$$
Then the tension field $\tau(f_1)$ satisfies
$$\tau(f_1)=t^2\cdot h(x)\cdot\frac{\partial^2 p_1}{\partial t^2}
-(n-2)\cdot t\cdot h(x)\cdot \frac{\partial p_1}{\partial t}=h(x)\cdot \tau(p_1).$$
This means that $f_1:\H^n\to\cn$ is harmonic if and only if $\tau(p_1)=0$,
or equivalently,
$$t^{-n}\cdot\tau(p_1)=t^{2-n}\cdot\frac{\partial^2 p_1}{\partial t^2}
+(2-n)\cdot t^{1-n}\cdot\frac{\partial p_1}{\partial t}
=\frac{\partial}{\partial t}(t^{2-n}\cdot\frac{\partial p_1}{\partial t})=0.$$
It then follows that there exists a complex constant $b_1$ such that
$$\frac{\partial p_1}{\partial t}=b_1(n-1)\cdot t^{n-2}$$
and integrating once again gives
$$p_1(t)=a_1+b_1\cdot t^{n-1}$$
for some constant $a_1\in\cn$.
\end{proof}

The above double integration leads to the following natural definition of the integral operator $I_n$.

\begin{definition}
Let $p:\rn^+\to\rn$ be a continuous function.  Then we define the integral operator $I_n$ by
$$I_n(p)(t)=\int t^{n-2}\cdot\bigl( \int t^{-n}\cdot p(t)dt+\alpha\bigr) dt+\beta.$$
Here $\alpha,\beta\in\cn$ are undetermined constants.
\end{definition}

\begin{theorem}\label{theo:upper}
Let the $n$-dimensional hyperbolic space $\H^n$ be modelled as the upper-half space i.e.
$\H^n=\rn^+\times\rn^{n-1}$.  Let $h:\rn^{n-1}\to\cn$ be a non-constant function harmonic
with respect to the Euclidean metric on $\rn^{n-1}$ and $p_r:\rn^+\to\cn$ be given by
$$p_r(t)=(a_r+b_r\cdot t^{n-1})\cdot\log (t)^{r-1},$$
where $(a_r,b_r)\in\cn^2$ is non-zero and $r\ge 1$. Then the function $f_r:\H^n\to\cn$
with $$f_r(t,x)=p_r(t)\cdot h(x)$$ is proper $r$-harmonic on $\H^n$.
\end{theorem}

\begin{proof}
We are assuming that $h:\rn^{n-1}\to\cn$ is a harmonic function with respect to
the Euclidean metric on $\rn^{n-1}$ so for each $k\in\zn^+$ the tension field $\tau^k(f_1)$ satisfies
$$\tau^k(f_1(x,t))=h(x)\cdot\tau^k(p_r(t)).$$
We have seen in Theorem \ref{theo:basic-step} that if $p_0(t)=0$ then $p_1(t)=I_n(p_0)(t)$ is of the form
$$p_1(t)=a_1+b_1\cdot t^{n-1},$$
where $a_1,b_1\in\cn$.  This implies that the function $f_1(t,x)=p_1(t)\cdot h(x)$ is proper $1$-harmonic
if and only if $(a_1,b_1)\in\cn^2$ is non-zero.

For the next step, it is easily seen that
$$I_n(p_1)(t)=q_{21}(t)+q_{22}(t),$$
where
$$q_{21}(t)=\frac 1{(n-2)^2}\cdot(\alpha (n-1)^2+((n-1)\beta-b_1)t^{n-1}),$$
$$q_{22}(t)=-\frac 1{(n-1)}(a_1-b_1t^{n-1})\cdot\log t.$$
and furthermore $\tau(q_{21})=0$.  This implies that for any non-zero $(a_2,b_2)\in\cn^2$ the function
$$p_2(t)=(a_2+b_2\cdot t^{n-1})\cdot\log t$$ is proper $2$-harmonic.  This immediately tells us that
$f_2:\H^n\to\cn$ of the form $f_2(t,x)=h(x)\cdot p_2(t)$ is proper 2-harmonic on $\H^n$ if and only if
$(a_2,b_2)\in\cn^2$ is non-zero.

This process can now be repeated and the result follows by induction.
\end{proof}

\begin{example}
Let $h:\rn^3\to\cn$ be a non-constant function harmonic
with respect to the Euclidean metric on $\rn^3$ and $p_2:\rn^+\to\cn$ be given by
$$p_2(t)=(a_2+b_2\cdot t^3)\cdot\log (t),$$
where $(a_2,b_2)\in\cn^2$ is non-zero. Further let the function $f_2:\H^4\to\cn$
be defined by $f_2(t,x)=p_2(t)\cdot h(x)$.
Then the tension fields $\tau(f_2)$ and $\tau^2(f_2)$ satisfy
$$\tau(f_2)=-3(a_2-b_2\cdot t^3)\cdot h(x)\ \ \text{and}\ \ \tau^2(f_2)=0.$$
This shows that $f:\H^4\to\cn$ is proper biharmonic.
\end{example}

\section{The hyperbolic one-sheeted hyperboloid $\hyp^n$}

We have already constructed a wealth of $r$-harmonic function on the $n$-dimensional
hyperbolic space $\H^n$ modelled as the upper half space.  We are now interested
in extending our constructions to the $n$-dimensional sphere $\s^n$ via a general
duality principle developed in \cite{Gud-Sve-1}.  For this we need to
understand our new constructions as functions from the hyperbolic space $\hyp^n$
modelled as the one-sheeted hyperboloid in Minkowski space.
\vskip .2cm

Let $M^{n+1}$ be the standard $(n+1)$-dimensional Minkowski space equipped with
its Lorentzian metric
$$
(x,y)_L=-x_0y_0+\sum_{k=1}^nx_ky_k.
$$
Bounded by the light cone, the open set
$$
U^{n+1}=\{y\in M^{n+1}|\ (y,y)_L<0\ \text{and}\ y_0>0\}
$$
contains the $n$-dimensional hyperbolic space
$$
\hyp^n=\{(y_0,y_1,\dots ,y_n)\in M^{n+1}|\ (y,y)_L=-1\ \text{and}\ y_0>0\}.
$$
Let $\pi:U^{n+1}\to \hyp^n$ be the radial projection given by
$$
\pi:y\mapsto \frac{y}{\sqrt{-(y,y)_L}}.
$$
This is a harmonic morphism and its dilation satisfies $\lambda^{-2}(y)=-|y|^2_L$,
see \cite{Gud-7}.  This means that for this situation we have the following special
version of Proposition \ref{prop:lift-tension}.

\begin{proposition}\label{prop:harmonic-morphism-Hn}
Let $\pi:U^{n+1}\to\hyp^n$ be the submersive harmonic morphism given by
$$\pi:y\mapsto \frac{y}{\sqrt{-(y,y)_L}}.$$
Further let $f:\hyp^n\to\cn$ be a smooth function and
$\hat f:U^{n+1}\to\cn$ be the composition $\hat f=f\circ\pi$.
Then the tension fields of $f$ satisfy
$$
\tau(f)\circ\pi=-|y|^2_L\cdot \tau(\hat f)\ \ \text{and}\ \
\tau^r(f)\circ\pi=-|y|^2_L\cdot\tau(-|y|^2_L\cdot\tau^{(r-1)}(\hat f))
$$
for all positive integers $r\ge 2$.
\end{proposition}

\begin{remark}
It should be noted that in the Minkowski space $M^{n+1}$ the
tension field is the classical the wave operator $\Box$ of d'Alembert given by
$$
\Box (\hat f) = -\frac{\partial^2\hat f}{\partial y_0^2}
+\frac{\partial^2\hat f}{\partial y_1^2}+\cdots +\frac{\partial^2\hat f}{\partial y_n^2}.
$$
\end{remark}

It is a classical fact that the map $\Psi:(\hyp ^n,ds^2_L)\to (\H^n,ds^2)$ given by
$$\Psi:(y_0,y_1,\dots,y_n)\mapsto 2\cdot (\frac{1}{y_0+y_1},\frac{y_2}{y_0+y_1},\dots,\frac{y_n}{y_0+y_1})$$
is an isometry between the two different models of the $n$-dimensional
hyperbolic space.  For this see for example \cite{Can-Flo-Ken-Par}.
This means that the composition $\Phi=\Psi\circ\pi:(U^{n+1},ds^2_L)\to (H^n,ds^2)$ satisfies
$$\Phi:(y_0,y_1,\dots,y_n)\mapsto (\frac{2\sqrt{y_0^2-y_1^2-\cdots -y_n^2}}{y_0+y_1},\frac{2y_2}{y_0+y_1},\dots,\frac{2y_n}{y_0+y_1}).$$
We now have the following result corresponding to Theorem \ref{theo:upper}.

\begin{theorem}\label{theo:hyperboloid}
Let the $n$-dimensional hyperbolic space $\hyp ^n$ be modelled as the one-sheeted hyperboloid
in the Minkowski space $M^{n+1}$.  Let $h:\rn^{n-1}\to\cn$ be a non-constant function harmonic
with respect to the Euclidean metric on $\rn^{n-1}$ and $p_r:\rn^+\to\cn$ be given by
$$p_r(t)=(a_r+b_r\cdot t^{n-1})\cdot\log (t)^{r-1},$$
where $(a_r,b_r)\in\cn^2$ is non-zero and $r\ge 1$. Then the function $f:\hyp ^n\to\cn$
with $$f(y_0,y_1,\dots,y_n)=p_r(\frac{2\sqrt{y_0^2-y_1^2-\cdots -y_n^2}}{y_0+y_1})\cdot h(\frac{2y_2}{y_0+y_1},\dots,\frac{2y_n}{y_0+y_1})$$ is proper $r$-harmonic on $\hyp ^n$.
\end{theorem}

\begin{proof}
This is a direct consequence of Theorem \ref{theo:upper}, Proposition
\ref{prop:harmonic-morphism-Hn}
and the fact that $\Psi:(\hyp ^n,ds^2_L)\to (\H^n,ds^2)$ is an isometry.
\end{proof}

\section{The $n$-dimensional sphere $\s^n$.}

Let the $(n+1)$-dimensional real vector space $\rn^{n+1}$ be equipped with its
standard Euclidean scalar product $(\cdot ,\cdot )$ satisfying
$$
(x,y)=x_1y_1+x_2y_2+\cdots +x_{n+1}y_{n+1}.
$$
Then the $n$-dimensional round unit sphere $\s^n$ in $\rn^{n+1}$ is given by
$$
\s^n=\{(y_1,y_2,\dots,y_{n+1})\in\R^{n+1}|\ y_1^2+y_2^2+\cdots +y_{n+1}^2=1\}.
$$
The radial projection $\pi:\R^{n+1}\setminus\{0\}\to \s^n$ with $\pi:y\mapsto y/|y|$
is a well-known harmonic morphism and its dilation satisfies $\lambda^{-2}(y)=|y|^2$.
This means that for this situation we have the following special version of Proposition
\ref{prop:lift-tension}.

\begin{proposition}\label{prop:harmonic-morphism-Sn}
Let $\pi:\rn^{n+1}\setminus\{0\}\to\s^n$ be the submersive harmonic morphism given by
$$\pi:y\mapsto \frac{y}{|y|}.$$
Further let $W$ be an open subset of $\s^n$, $f:W\to\cn$ be a smooth function and
$\hat f:\pi^{-1}(W)\subset\rn^{n+1}\setminus\{0\}\to\cn$ be the composition $\hat f=f\circ\pi$.
Then the tension fields of $f$ satisfy
$$
\tau(f)\circ\pi=|y|^2\cdot \tau(\hat f)\ \ \text{and}\ \
\tau^r(f)\circ\pi=|y|^2\cdot\tau(|y|^2\cdot\tau^{(r-1)}(\hat f))
$$
for all positive integers $r\ge 2$.
\end{proposition}

\begin{remark}
It should be noted that in the $(n+1)$-dimensional Euclidean space $\rn^{n+1}$ the
tension field is the classical Laplace operator $\Delta$ given by
$$
\Delta (\hat f) = \frac{\partial^2\hat f}{\partial y_1^2}
+\frac{\partial^2\hat f}{\partial y_2^2}+\cdots
+\frac{\partial^2\hat f}{\partial y_{n+1}^2}.
$$
\end{remark}

\begin{theorem}\label{theo:sphere}
Let $\s^n$ be the round unit sphere in the standard $(n+1)$-dimensional Euclidean
space $\rn^{n+1}$.  Let $h:\rn^{n-1}\to\cn$ be a non-constant function harmonic
with respect to the Euclidean metric on $\rn^{n-1}$ and $p_r:\rn^+\to\cn$ be given by
$$p_r(t)=(a_r+b_r\cdot t^{n-1})\cdot\log (t)^{r-1},$$
where $(a_r,b_r)\in\cn^2$ is non-zero and $r\ge 1$. Then the function 
$f_r:W\to\cn$ defined on an open subset $W$ of $\s^n$
with $$f_r(y_1,\dots,y_n)=p^*_r(\frac{2|y|}{y_2+i\cdot y_1})
\cdot h^*(\frac{2y_3}{y_2+i\cdot y_1},\dots,\frac{2y_{n+1}}{y_2+i\cdot y_1})$$
is proper $r$-harmonic.  Here $p^*_r$ and $h^*$ are some local complex analytic 
extensions of $p_r$ and $h$, respectively.
\end{theorem}

\begin{proof}
This is a direct consequence of Theorem \ref{theo:hyperboloid}, Proposition
\ref{prop:harmonic-morphism-Sn} and a general duality principle for $r$-harmonic
functions on spheres $\s^n$ and hyperbolic spaces $\hyp^n$.  For this see either
Theorem 7.1 of \cite{Gud-Sve-1} or Theorem 8.1 in \cite{Gud-Mon-Rat-1}.
\end{proof}

\end{document}